\documentclass[reqno, 12pt,a4letter]{amsart}
\usepackage{amsmath,amsxtra,amssymb,latexsym, amscd,amsthm}
\usepackage{graphicx,color}
\usepackage[utf8]{inputenc}
\usepackage[mathscr]{euscript}
\usepackage{mathrsfs}
\usepackage[english]{babel}
\usepackage{color}
\newtheorem {theorem}{Theorem}[section]
\newtheorem {corollary}{Corollary}[section]

\newtheorem {lemma}{Lemma}[section]
\newtheorem {example}{Example}[section]

\newtheorem {remark}{Remark}[section]

\def\ees{{\accent"5E e}\kern-.385em\raise.2ex\hbox{\char'23}\kern-.08em}
\def\EES{{\accent"5E E}\kern-.5em\raise.8ex\hbox{\char'23 }}
\def\ow{o\kern-.42em\raise.82ex\hbox{
\vrule width .12em height .0ex depth .075ex \kern-0.16em \char'56}\kern-.07em}
\def\OW{O\kern-.460em\raise1.36ex\hbox{
\vrule width .13em height .0ex depth .075ex \kern-0.16em \char'56}\kern-.07em}

\title[]{ON TANGENT CONES AT INFINITY OF ALGEBRAIC VARIETIES}

\author{C\^{O}NG-TR\`{I}NH L\^{E}$^\dagger$}
\address{Department of Mathematics, Quy Nhon University, 170 An Duong Vuong, Quy Nhon, Binh Dinh, Vietnam}
\email{lecongtrinh@qnu.edu.vn}

\author{TI\EES N-S\OW N PH\d{A}M$^{\ddagger}$}
\address{Department of Mathematics, University of Dalat, 1 Phu Dong Thien Vuong, Dalat, Vietnam}

\email{sonpt@dlu.edu.vn}
 
\keywords{Algebraic variety, tangent cone, tangent cone at infinity}

\subjclass[2010]{14Axx, 32Sxx}

\date{ \today}

\begin{document}
\maketitle

\begin{abstract} 
In this paper, we establish the following version at infinity of Whitney's theorem \cite{Whitney1965-1, Whitney1965-2}: {\em Geometric and algebraic tangent cones at infinity of complex algebraic varieties coincide.} The proof of this fact is based on a geometric characterization of geometric tangent cones at infinity and using the global \L ojasiewicz inequality with explicit exponents for complex algebraic varieties. We also show that tangent cones at infinity of complex algebraic varieties can be computed using Gr\"obner bases.
\end{abstract}

\section{Introduction}
Let $V \subset (\mathbb{C}^{n}, 0)$ be an analytic variety in some neighborhood of the origin $0$ in $\mathbb{C}^{n}.$ In his seminal papers \cite{Whitney1965-1, Whitney1965-2}, Whitney showed (among other results) that geometric and algebraic tangent cones of $V$ at $0$  coincide. The aim of this paper is to give a version at infinity of this result for complex algebraic varieties.

In order to state our main result, we need some notation. Let $V \subset \mathbb{C}^{n}$ be an algebraic variety. The {\em geometric tangent cone $C_{g, \infty}(V)$ of $V$ at infinity} is defined by the set of ``tangent'' vectors, in the sense that $v \in C_{g, \infty}(V)$ if, and only if, there exist a sequence $\{x_k\}_{k \in \mathbb{N}} \subset V, \|x_k\| \to \infty,$ and a sequence of numbers $\{t_k\}_{k \in \mathbb{N}} \subset \mathbb{C}$ such that $t_k x_k \to v$.   By the {\em algebraic tangent cone $C_{a, \infty}(V)$ of $V$ at infinity} we mean the set 
$$C_{a, \infty}(V) := \left \{ v \in \mathbb{C}^{n}\, | \,  f^*(v) = 0  \mbox{ for all } f\in \bold{I}(V)\right\},$$
where $\bold{I}(V)$ denotes the ideal defining $V$, and for each polynomial $f\in \bold{I}(V)$, $f^*$ 
denotes  its  homogeneous component of highest degree. 

The main result of this paper can be formulated as follows.

\begin{theorem} \label{Whitney-at-infinity}
Let $V \subset \mathbb{C}^n$ be an algebraic variety. Then $C_{a, \infty}(V) = C_{g, \infty}(V).$
\end{theorem}

\begin{example}{\rm
Let $V := \{(x, y) \in \mathbb{C}^2 \, | \, x^2 - y^3 = 0\}.$ We have 
$$C_{a, \infty}(V) = C_{g, \infty}(V) = \{ y = 0\}.$$
}\end{example}

\begin{remark}{\rm 
The proof for Theorem~\ref{Whitney-at-infinity} is different from the one given in \cite[Theorem 10.6]{Whitney1965-2}. In fact, the proof of the inclusion $C_{g, \infty}(V) \subseteq C_{a, \infty}(V)$ is rather straightforward (Lemma~\ref{Lemma3}), while the proof of the converse one will follow from a geometric characterization of geometric tangent cones at infinity (Lemma~\ref{Lemma4}) and the global \L ojasiewicz inequality with explicit exponents for algebraic varieties (Lemma~\ref{Lemma5}). 
}\end{remark}

The paper is organized as follows. The proof of Theorem~\ref{Whitney-at-infinity} will be given in Section~\ref{section2}.
Then in Section~\ref{section3}, using Gr\"obner bases, we compute tangent cones at infinity of algebraic varieties.

\section{Proof of Theorem~\ref{Whitney-at-infinity}} \label{section2}

Let $V \subset \mathbb{C}^n$ be an algebraic variety. By definition, we can verify that the sets $C_{g,\infty}(V)$ and $C_{a,\infty}(V)$ are closed \emph{cones} in $\mathbb C^n$ with the vertex at the origin $0 \in \mathbb{C}^{n}$, i.e. for every element $v \in C_\infty(V)$  and for every $\lambda  \in \mathbb C$, we have $\lambda v \in C_\infty(V)$, where $C_\infty(V)$ denotes for both $C_{g,\infty}(V)$ and $C_{a,\infty}(V)$. Moreover,  the following properties of these cones are obtained easily from their definitions.
\begin{lemma} \label{Lemma1}
For any two algebraic varieties $V$ and $W$ in $\mathbb{C}^{n},$ the following properties hold:
\begin{enumerate}
\item[(i)] $C_\infty(W) \subseteq C_\infty(V)$ if $W\subseteq V;$
\item[(ii)] $C_\infty(V \cup W) = C_\infty(V)\cup C_\infty(W);$
\item[(iii)]  $C_\infty(V\cap W) \subseteq C_\infty(V) \cap C_\infty(W).$
\end{enumerate}
\end{lemma}

\begin{remark}\rm 
In general we cannot replace ``$\subseteq$''  by ``$=$'' in  (iii).  Indeed, let us  consider the algebraic varieties $V$ and $W$ in $\mathbb{C}^2$ defined  respectively by $x = 0$ and $y - x^2 = 0.$
Then $C_\infty(V) = C_\infty(W) = \{x = 0\}$, while $C_\infty(V\cap W) = \{0\}.$ Hence in this case, $C_\infty(V\cap W) \varsubsetneq C_\infty(V) \cap C_\infty(W).$
\end{remark}

The following statement shows that tangent vectors to $V$ at infinity may be defined by analytic curves  (compare \cite[Proposition~2]{Oshea-2004} and \cite[Theorem~11.8]{Whitney1965-2}).
\begin{lemma} \label{Lemma2}
Let $V \subset \mathbb{C}^{n}$ be an algebraic variety. Then, for each nonzero vector $v \in C_{g, \infty}(V),$ there exists an analytic curve $\varphi \colon (0, \epsilon)  \rightarrow \mathbb{C}^n$ $(\epsilon > 0)$ with $\varphi(s) \in V$ for all $s \in (0, \epsilon),$ such that 
$$\lim_{s \to 0^+} \| \varphi(s)\| = + \infty \quad \textrm{ and } \quad \lim_{s \to 0^+} \frac{\varphi(s)}{\|\varphi(s)\|} \ = \ -\lim_{s \to 0^+} \frac{\varphi'(s)}{\|\varphi'(s)\|} = \frac{v}{\|v\|}.$$
\end{lemma}
\begin{proof}
Indeed, let $v$ be a  nonzero vector in $C_{g, \infty}(V).$ By definition, there exist a sequence $\{x_k\}_{k \in \mathbb{N}} \subset V, \|x_k\| \to \infty,$ and a sequence of numbers $\{t_k\}_{k \in \mathbb{N}} \subset \mathbb{C}$ such that $t_k x_k \to v.$ In view of the  Curve Selection Lemma at infinity (see \cite{Dinh2014, Milnor1968}), there exist analytic curves $\psi := (\psi_1, \ldots, \psi_n) \colon (0, \epsilon)  \rightarrow \mathbb{C}^n$ and $t \colon (0, \epsilon)  \rightarrow \mathbb{C},$ for some $\epsilon > 0,$ such that
\begin{enumerate}
\item [(a)] $\psi(s) \in V$ for all $s \in (0, \epsilon);$
\item [(b)] $\lim_{s \to 0^+} \| \psi(s)\| = + \infty;$
\item [(c)] $\lim_{s \to 0^+} t(s) \psi(s) = v.$
\end{enumerate}
Since $v \ne 0,$ we have that $t(s) \not \equiv 0,$ and so we can write
\begin{eqnarray*}
t(s) &=& a s^{p} + \textrm{ higher terms in } s,
\end{eqnarray*}
for some $a \in \mathbb{C},$ with $a \ne 0,$ and $p \in \mathbb{Z}.$ Let $J := \{j \, | \, \psi_j(t) \not\equiv 0\} \ne \emptyset.$ For each $j \in J$ we can write
\begin{eqnarray*}
\psi_j(s) &=& b_j  s^{q_j} + \textrm{ higher terms in } s,
\end{eqnarray*}
for some $b_j \in \mathbb{C},$ with $b_j \ne 0,$ and $q_j \in \mathbb{Z}.$ It follows from conditions (b) and (c) that
\begin{eqnarray*}
p &=& - \min_{j \in J} q_j > 0
\end{eqnarray*}
and
\begin{eqnarray*}
v & = &
\begin{cases}
{ab_j}  & \textrm{ if  } j \in J \textrm{ and } q_j = -p, \\
0 & \textrm{ otherwise.}
\end{cases}
\end{eqnarray*}

Let $\lambda \in \mathbb{C}$ be such that $\lambda^{-p} = a$ and we may regard $\psi$ as an analytic curve from the disc $\{s \in \mathbb{C} \, | \, |s| < \epsilon\}$ to $\mathbb{C}^n.$ Define the curve $\varphi \colon (0, \epsilon/|\lambda|)  \rightarrow \mathbb{C}^n$ by $\varphi(s) := \psi(\lambda s).$ Then, by condition~(a), $\varphi(s) \in V$ for all $s \in (0, \epsilon/|\lambda|).$ Furthermore, we have
$$\|\varphi(s)\| = c s^{-p} + \textrm{ higher terms in } s$$
for some real number $c > 0.$ Finally, a direct computation shows that
\begin{eqnarray*}
\lim_{s \to 0^+} \frac{\varphi_j(s)}{\|\varphi(s)\|} \ = \
-\lim_{s \to 0^+} \frac{\varphi_j'(s)}{\|\varphi'(s)\|} & = &
\begin{cases}
\frac{ab_j}{c}  & \textrm{ if  } j \in J \textrm{ and } q_j = -p, \\
0 & \textrm{ otherwise.}
\end{cases}
\end{eqnarray*}
This completes the proof.
\end{proof}

\begin{remark}{\rm 
By Lemma~\ref{Lemma2}, it will not alter the definition of the cone $C_{g, \infty}(V)$ if we require that the $t_k$ be real and positive. In particular, $v \in C_{g, \infty}(V) \setminus \{0\}$ if, and only if, 
there exists a sequence $\{x_k\}_{k \in \mathbb{N}} \subset V$ such that 
$$\lim_{k \to \infty}  \|x_k\| = +\infty \quad \textrm{ and } \quad  \lim_{k \to \infty} \frac{x_k}{\|x_k\|} = \frac{v}{\|v\|}.$$
}\end{remark}

\begin{lemma} \label{Lemma3}
Let $V \subset \mathbb{C}^n$ be an algebraic variety. Then $C_{g,\infty}(V)\subseteq C_{a,\infty}(V).$
\end{lemma}

\begin{proof} 
We first consider the case where $V$ is a hypersurface defined by a polynomial  $f \in \mathbb{C}[x_1,  \ldots,x_n].$ We can write 
$$f = f_d + f_{d - 1} + \cdots + f_0,$$
where $d := \deg f$ and each $f_i$ is homogeneous polynomial of degree $i$.  

Take any non-zero vector $v \in C_{g, \infty}(V).$ By Lemma~\ref{Lemma2}, there exists a sequence $\{x_k\}_{k \in \mathbb{N}} \subset V$ such that 
$$\lim_{k \to \infty}  \|x_k\| = +\infty \quad \textrm{ and } \quad  \lim_{k \to \infty} \frac{x_k}{\|x_k\|} = \frac{v}{\|v\|}.$$
For $k$ sufficiently large we have $\|x_k\| > 0$ and hence
\begin{eqnarray*}
0 &=& f(x_k) \ = \ f\left(\|x_k\| \frac{x_k}{\|x_k\|}\right) \\
&=&
\|x_k\|^d \left[f_d \left(\frac{x_k}{\|x_k\|}\right) + 
\frac{1}{\|x_k\|} f_{d - 1} \left(\frac{x_k}{\|x_k\|}\right) + \cdots + 
\frac{1}{\|x_k\|^d} f_{0} \left(\frac{x_k}{\|x_k\|}\right) \right],
\end{eqnarray*}
which implies that
\begin{eqnarray*}
0 &=& f_d \left(\frac{x_k}{\|x_k\|}\right) + 
\frac{1}{\|x_k\|} f_{d - 1} \left(\frac{x_k}{\|x_k\|}\right) + \cdots + 
\frac{1}{\|x_k\|^d} f_{0} \left(\frac{x_k}{\|x_k\|}\right).
\end{eqnarray*}
Letting $k \to \infty$ yields $f_d(v) = 0$, i.e. $v\in C_{a,\infty}(V).$

Now let  $V$ be  any algebraic variety in $\mathbb C^n$ and $v \in C_{g, \infty}(V).$ Take any $f \in \bold{I}(V).$ We have 
$V \subseteq \mathbf{V}(f) := \{x \in \mathbb{C}^n \, | \, f(x) = 0 \},$ and so, by Lemma~\ref{Lemma1}(i), $C_{g, \infty}(V) \subseteq C_{g, \infty}(\bold{V}(f)).$ 
It follows from the hypersurface case considered above that $v\in C_{a,\infty}(\bold{V}(f)).$ Note that $C_{a,\infty}(\bold{V}(f)) = \{w \in \mathbb{C}^n \, | \, f_d(w) = 0\}.$ Therefore, $f_d(v) = 0.$ Since $f$ is arbitrary, $v\in C_{a,\infty}(V).$ The proof is complete. 
\end{proof}

In order to show the reverse inclusion $C_{g,\infty}(V) \supseteq C_{a,\infty}(V),$ we need the following geometric characterization of the cone $C_{g, \infty}(V).$ 

\begin{lemma} \label{Lemma4}
Let $V \subset \mathbb{C}^n$ be an algebraic variety and let $v$ be a non-zero vector in $\mathbb{C}^{n}.$ Then $v \in C_{g, \infty}(V)$ if and only if there exists a sequence $\{t_k\}_{k \in \mathbb{N}}$ of positive real numbers such that 
$$\lim_{k \to \infty} t_k = +\infty \quad \textrm{ and } \quad \lim_{k \to \infty} \frac{\mathrm{dist}(t_k v, V)}{t_k} = 0,$$
where $\mathrm{dist}(x, V)$ denotes  the  usual Euclidean distance from a point $x \in \mathbb{C}^{n}$ to the variety $V.$
\end{lemma}

\begin{proof}
$\Rightarrow$ Assume that $v \in C_{g, \infty}(V).$ By Lemma~\ref{Lemma2}, there exists a sequence $\{x_k\}_{k \in \mathbb{N}} \subset V$ such that
$$\lim_{k \to \infty} \|x_k\| = +\infty \quad \textrm{ and } \quad \lim_{k \to \infty} \frac{x_k}{\|x_k\|} = \frac{v}{\|v\|}.$$
Then we may assume that $t_k := \frac{\|x_k\|}{\|v\|} > 0$ for all $k \in \mathbb{N}.$ We have $\lim_{k \to \infty} t_k = + \infty$ and 
$$\frac{\mathrm{dist}(t_k v, V)}{t_k} \le  \frac{\|t_k v - x_k\|}{t_k} = \left\|v - \frac{\|v\|}{\|x_k\|}x_k \right\| = \|v\|\left\| \frac{v}{\|v\|} - \frac{x_k}{\|x_k\|}\right\|.$$ 
Therefore, 
$$ \lim_{k \to \infty} \frac{\mathrm{dist}(t_k v, V)}{t_k} = 0.$$

$\Leftarrow$ Conversely, assume there exists a sequence $\{t_k\}_{k \in \mathbb{N}}$ of positive real numbers such that
$$\lim_{k \to \infty} t_k = +\infty \quad \textrm{ and } \quad \lim_{k \to \infty} \frac{\mathrm{dist}(t_k v, V)}{t_k} = 0.$$

For each $k \in \mathbb{N}$,  let  $x_k$ be a point  in $V$ such that  $\mathrm{dist}(t_kv, V) = \|t_kv - x_k\|$ (such a point does always exist because $V$ is a closed set in the usual topology on $\mathbb{C}^{n}$).  Then
$$\lim_{k \to \infty} \left \|v - \frac{x_k}{t_k}\right \| 
= \lim_{k \to \infty} \frac{\|t_k v - x_k\|}{t_k} = 
\lim_{k \to \infty} \frac{\mathrm{dist}(t_k v, V)}{t_k} = 0.$$
Hence $\lim_{k \to \infty} \frac{x_k}{t_k} = v$ and so $\lim_{k \to \infty} \|x_k\| = +\infty$ because $\lim_{k \to \infty} t_k = + \infty.$ Furthermore we have
$$\lim_{k \to \infty} \frac{\|x_k\|}{t_k} = \|v\|.$$ 
Therefore
$$\lim_{k \to \infty} \frac{x_k}{\|x_k\|} \ = \ 
\lim_{k \to \infty} \left( \frac{x_k}{t_k} \times \frac{t_k}{\|x_k\|} \right) \ = \ 
\lim_{k \to \infty} \frac{x_k}{t_k} \times \lim_{k \to \infty}  \frac{t_k}{\|x_k\|} \ = \ \frac{v}{\|v\|}.$$
The proof is complete.
\end{proof}

Let $V\subset \mathbb C^n$ be an algebraic variety of dimension $k.$ It is well-known that, for any generic $(n - k)$-dimensional linear subspace $P$ in $\mathbb{C}^{n},$  the set $V\cap P$ is finite, and its cardinality is called the {\em degree} of $V$ and denoted by $\deg(V)$. In particular, if $V =  \mathbf{V}(f) := \{x \in \mathbb{C}^n \, | \, f(x) = 0 \}$ is a hypersurface defined by a complex polynomial $f$ of degree $d,$ then $\deg(V)=d.$
The following lemma will be useful in the sequel.

\begin{lemma}[{\cite[Lemma~8]{JKS1992}}] \label{Lemma5}  
Let $V \subset \mathbb{C}^{n}$ be an irreducible algebraic variety of degree $d$. Then there are finitely many polynomials $g_i, i=1, \ldots, s$, of degree at most $d$ vanishing on $V$ and a constant $c > 0$ such that
\begin{eqnarray*}
\mathrm{dist}(x, V)^d & \leq & c \max_{i = 1, \ldots, s} \{|g_i(x)|\} \quad 
\textrm{ for all } \quad x \in \mathbb{C}^{n}.
\end{eqnarray*}
\end{lemma}

We are now ready to complete the proof of Theorem~\ref{Whitney-at-infinity}.

\begin{proof}[Proof of Theorem \ref{Whitney-at-infinity}]  
By Lemma~\ref{Lemma3}, we have $C_{g, \infty}(V) \subseteq C_{a, \infty}(V).$ Hence, it remains to show the reverse inclusion.

We first remark that it suffices to prove the inclusion $C_{a, \infty}(V) \subseteq C_{g, \infty}(V)$ for the case where the variety $V$ is irreducible. Indeed, if  $V_1, \ldots, V_r$ are irreducible components of $V$, and assume that  we have proved that $C_{g, \infty}(V_i) = C_{a, \infty}(V_i)$ for every $i$, then by Lemma~\ref{Lemma1}(ii) we have 
$$ C_{g, \infty}(V) = \bigcup_{i=1}^{r}C_{g, \infty}(V_i) = \bigcup_{i=1}^{r}C_{a, \infty}(V_i) = C_{a, \infty}(V). $$

Therefore, we  may now assume  that $V$ is irreducible of degree $d$.  By Lemma~\ref{Lemma5}, there exist  finitely many polynomials $g_i, i = 1, \ldots, s$, of degree at most $d$ vanishing on $V$ and a constant $c > 0$ such that for all $x \in \mathbb{C}^{n}$ we have 
\begin{eqnarray*}
\mathrm{dist}(x,V)^d & \leq & c \max_{i=1, \ldots, s} \{|g_i(x)|\}.
\end{eqnarray*} 

Let $v$ be a non-zero vector in $C_{a, \infty}(V).$ Since $V\subseteq \bold{V}(g_i)$ for every $i=1, \ldots, s,$ by Lemma~\ref{Lemma1}(i) we have $C_{a, \infty}(V) \subseteq C_{a, \infty}(\bold{V}(g_i)),$ and so $v\in C_{a, \infty}(\bold{V}(g_i)).$ Hence
\begin{eqnarray} \label{pt-gi}
(g_i)_{d_i}(v) &=& 0, \quad \textrm{ for all } \quad i = 1, \ldots, s,
\end{eqnarray}
where  $(g_i)_{d_i}$ denotes the homogeneous component of $g_i$ of degree $d_i := \deg g_i.$ 

Now let  $\{t_k\}_{k \in \mathbb{N}}$ be a sequence of positive real numbers such that $\lim_{k \to \infty} t_k = +\infty.$ It follows from (\ref{pt-gi}) that for every $i=1, \ldots, s,$ there exists a positive constant $c_i$ such that for all $k$ sufficiently large,
\begin{eqnarray*}
|g_i(t_kv)| & \leq & c_i t_k^{d_i-1}.
\end{eqnarray*}
Since $d_i := \deg g_i \leq d$ for $i=1, \ldots, s$, we obtain for all $k$ sufficiently large,
\begin{eqnarray*}
\max_{i=1, \ldots, s} \{ |g_i(t_kv)| \} & \leq & c' t_k^{d-1},
\end{eqnarray*}
where $c' :=\max_{i=1, \ldots,s}c_i > 0.$ Therefore
\begin{eqnarray*}
\frac{\mathrm{dist}(t_kv,V)}{t_k} & \leq & \frac{\left(c \max_{i=1, \ldots, s} \{|g_i(t_kv)|\} \right)^{\frac{1}{d}}}{t_k} \ \leq \ \left(c' c \right)^{\frac{1}{d}}t_k^{\frac{d-1}{d}-1} \ = \ 
\left (c' c \right)^{\frac{1}{d}}t_k^{\frac{-1}{d}}.
\end{eqnarray*}
By letting $k \to  \infty$, we get
\begin{eqnarray*}
\lim_{k \to  \infty} \frac{\mathrm{dist}(t_kv,V)}{t_k} & = & 0.
\end{eqnarray*}
Hence $v \in C_{g, \infty}(V)$ by Lemma~\ref{Lemma4}. The proof of the theorem is complete. 
\end{proof}

\section{Computations} \label{section3}
In this section, given an algebraic variety $V\subset \mathbb C^n,$ we shall compute the tangent cone at infinity $C_\infty(V):=C_{a,\infty}(V)=C_{g,\infty}(V)$ of $V$ by using Gr\"obner bases.

If  $\bold{I}(V)=\left<f\right> \subset \mathbb C[x_1, \ldots,x_n]$, then it is easy to see that 
$$C_{\infty}(V)=\bold{V}(f^*) := \{v\in \mathbb C^n \, | \, f^*(v)=0\}.$$
(Recall that $f^*$ stands for the homogeneous component of highest degree of $f$.)
However, if $\bold{I}(V)=\left<f_1, \ldots, f_r\right>$ has more  generators, then it need \emph{not} follow that 
$$C_{\infty}(V)=\bold{V}(f_1^*, \ldots,f_r^*):= \{v\in \mathbb C^n \, | \, f_1^*(v) = \cdots = f_r^*(v) = 0\}.$$
For example, let $V \subset \mathbb C^3$ be defined by the  ideal generated by polynomials 
$$f_1 :=xy \quad  \textrm{ and } \quad f_2 := z(x^3-y^2+z^2).$$
We have $f_1^*=xy, f_2^*=x^3z$, and 
$$\bold{V}(f_1^*,f_2^*)=\{x=0\} \cup \{y=z=0\}.$$
Consider the polynomial $f \in \mathbb{C}[x, y, z]$ defined  by
$$f(x, y, z) :=yz(y^2-z^2)=zx^2f_1-yf_2.$$
We see that $f^*=f$ vanishes on $C_{\infty}(V)$, however $f^*$ does not vanish on $\bold{V}(f_1^*,f_2^*)$. It follows that $C_{\infty}(V)\not =\bold{V}(f_1^*,f_2^*).$
 
We can overcome this difficulty by using an appropriate Gr\"obner basis for the ideal defining $V$. In order to  compute the ideal defining the cone $C_\infty(V)$ we need the following notation.

For a polynomial $f\in \mathbb C[x_1, \ldots,x_n]$ of degree $d$, its {\em homogenization} $f^h$ is a polynomial in $\mathbb C[x_0,x_1, \ldots,x_n]$ defined by 
$$f^h(x_0,x_1, \ldots,x_n) := x_0^df \left(\frac{x_1}{x_0}, \ldots,\frac{x_n}{x_0}\right), $$
where $x_0$ is a new variable. For an ideal $I$ in the ring $\mathbb C[x_1, \ldots,x_n]$, its {\em homogenization} is  the ideal 
$$ I^h:=\left<f^h \, | \, f\in I\right> \subset \mathbb C[x_0,x_1, \ldots,x_n].$$
For any $g\in \mathbb C[x_0,x_1, \ldots,x_n]$, let 
$$g|_{x_0=0} := g(0, x_1, \ldots,x_n) \in \mathbb C[x_1,\ldots,x_n]. $$
Following  \cite{GP2012} we define
$$I_\infty:=\left<g|_{x_0=0} \, | \, g \in I^h \right>\subset \mathbb C [x_1, \ldots, x_n]. $$

\begin{lemma} \label{part-at-infinity} 
Let $V\subset \mathbb C^n$ be an algebraic variety and let $I := \mathbf{I}(V)$. Then
$$C_{\infty}(V) = \mathbf{V}(I_\infty). $$
\end{lemma}

\begin{proof} 
Let us take any $x=(x_1, \ldots,x_n)\in C_{\infty}(V)$ and $f\in I_\infty$. Then there exists $g\in I$ such that $f= g^h|_{x_0=0}$. Decomposing $g$ into homogeneous components as
$$ g = g_0+g_1+\cdots+g_d, $$
where $d$ is the degree of $g$, we have 
$$ g^h=g_0x_0^d + g_1x_0^{d-1}+\cdots+g_{d-1}x_0+g_d \in \mathbb C[x_0,x_1,\ldots,x_n].$$
It follows that 
$$f=g^h|_{x_0=0}=g_d=g^*. $$
By definition of $C_\infty(V)=C_{a,\infty}(V),$ we have $f(x) = g^*(x) = 0,$ and so $x\in \bold{V}(I_\infty).$ Therefore, $C_{\infty}(V)\subseteq \bold{V}(I_\infty).$

Conversely, let us  take any $x\in \bold{V}(I_\infty)$ and  $f\in I$. Assume we decompose $f$ into homogeneous components as
$$f=f_0+f_1+\cdots+f_e, $$
where $e$ is the degree of $f.$ Then
$$f^h=f_0x_0^e+f_1x_0^{e-1}+\cdots+f_{e-1}x_0+f_e \in I^h. $$
It follows that
$$f_e = f^h |_{x_0=0} \in I_\infty.$$
By definition of $\bold{V}(I_\infty),$ therefore
 $f^*(x)=f_e(x)=0.$ It follows that $x\in C_{a,\infty}(V)=C_\infty(V)$. The proof is complete.
\end{proof}
By \cite[Lemma A.4.15]{GP2012}, the variety $\bold{V}(I_\infty)\subset \mathbb C^n$ equals to the so-called {\em part at infinity} of $V.$ Hence, it follows from Lemma \ref{part-at-infinity} and the remark after \cite[Lemma A.4.15]{GP2012} the following  computation for $C_\infty(V) = \mathbf{V}(I_\infty)$.

\begin{corollary}\label{coro-groebner} Let $V\subset \mathbb C^n$ be an algebraic variety. Let $\{g_1,\ldots,g_k\}$ be a Gr\"{o}bner basis for $\bold{I}(V)$ with respect to a degree ordering,  then 
$$C_\infty(V) = \bold{V}(g_1^h|_{x_0=0}, \ldots, g_k^h|_{x_0=0}). $$

\end{corollary}

\begin{example}{\rm 
Let us compute the tangent cone at infinity of the algebraic variety $V \subset \mathbb C^3$ defined by the ideal $I \subset \mathbb C[x,y,z]$ generated by polynomials $f_1 :=xy $ and $ f_2 := z(x^3-y^2+z^2)$ as considered at the beginning of this section. 

It is easy to compute a Gr\"{o}bner basis for $I$ with respect to a degree ordering is
$$\{g_1=xy, \quad g_2=x^3z-y^2z+z^3, \quad  g_3=yz(y^2-z^2)\}. $$
The homogenization of these polynomials by a new variable $t$ in the polynomial ring $\mathbb C[x,y,z,t]$:
 $$g_1^h=xy, \quad g_2^h=x^3z-y^2zt+z^3t, \quad  g_3^h=yz(y^2-z^2). $$
Substituting $t=0$, using Corollary \ref{coro-groebner}, we have 
$$ C_\infty(V)=\bold{V}(xy, x^3z, yz(y^2-z^2)),$$
which is a union of 5 lines through the origin in $\mathbb C^3$.
}\end{example}

\subsection*{Acknowledgment} 
The authors would like to express their gratitude to Prof. Gert-Martin Greuel for showing  them a quick procedure to compute the (algebraic) tangent cone at infinity, improving the one given in the original version of the paper. This research was partially performed while the second author had been visiting at Vietnam Institute for Advanced Study in Mathematics (VIASM). He would like to thank the Institute for hospitality and support.

\end{document}